\documentclass[12pt,reqno]{amsart}

\usepackage{amsmath,amsfonts,amssymb,amscd,amsthm,calc,enumerate}
\usepackage[english]{babel}
\usepackage{yfonts,color}

\newtheorem{theorem}{Theorem}[section]
\newtheorem{proposition}[theorem]{Proposition}

\newtheorem{corollary}[theorem]{Corollary}

\theoremstyle{definition}
\newtheorem{definition}[theorem]{Definition}
\newtheorem{remark}[theorem]{Remark}

\newcommand{\darrow}{\!\downarrow}
\newcommand{\uarrow}{\!\uparrow}

\renewcommand{\leq}{\leqslant}

\newcommand{\fa}{\forall}

\newcommand{\vph}{\varphi}

\newcommand{\A}{\mathcal{A}}
\newcommand{\K}{\mathcal{K}}
\newcommand{\LL}{\mathcal{L}}

\newcommand{\dom}{\mathrm{dom}}

\begin{document}

\title[Partial combinatory algebra and generalized numberings]
{Partial combinatory algebra and generalized numberings}

\author[H. P. Barendregt]{Henk Barendregt}
\address[Henk Barendregt]{Radboud University Nijmegen\\
Institute for Computing and Information Sciences\\
P.O. Box 9010, 6500 GL Nijmegen, the Netherlands.
} \email{henk@cs.ru.nl}

\author[S. A. Terwijn]{Sebastiaan A. Terwijn}
\address[Sebastiaan A. Terwijn]{Radboud University Nijmegen\\
Department of Mathematics\\
P.O. Box 9010, 6500 GL Nijmegen, the Netherlands.
} \email{terwijn@math.ru.nl}

\begin{abstract} 
Generalized numberings are an extension of Ershov's notion of 
numbering, based on partial combinatory algebra (pca) instead of 
the natural numbers. We study various algebraic properties of 
generalized numberings, relating properties of the numbering 
to properties of the pca. As in the lambda calculus, 
extensionality is a key notion here.   
\end{abstract}



\date{\today}

\maketitle

\section{Introduction}

A numbering is a surjective mapping $\gamma :\omega \rightarrow S$ 
from the natural numbers $\omega$ to a set $S$.  
The theory of numberings was started by Ershov in a series of 
papers, beginning with \cite{Ershov} and \cite{Ershov2}. 
Ershov studied the computability-theoretic properties of numberings, 
as generalizations of numberings of the partial computable functions. 
In particular, he called a numbering {\em precomplete\/} if 
for every partial computable unary function $\psi$ there exists a 
computable unary $f$ such that for every~$n$
\begin{equation} \label{precomplete}
\psi(n)\darrow \; \Longrightarrow \; \gamma(f(n))=\gamma(\psi(n)). 
\end{equation}
Following Visser, we say that
{\em $f$ totalizes $\psi$ modulo~$\gamma$}.
Ershov showed that Kleene's recursion theorem holds for arbitrary
precomplete numberings. 
Visser~\cite{Visser} extended this to his so-called 
``anti diagonal normalization theorem'' (ADN theorem). 
Another generalization of the recursion theorem is the famous
Arslanov completeness criterion \cite{Arslanov}, 
that extends the recursion theorem from computable functions to 
all functions bounded by an incomplete computably enumerable (c.e.) 
Turing degree.
Barendregt and Terwijn~\cite{BarendregtTerwijn} showed that 
Arslanov's result also holds for any precomplete numbering. 
In Terwijn~\cite{Terwijn} a joint generalization of Arslanov's 
completeness criterion and Visser's ADN theorem was proved. 
It is currently open whether this joint generalization also 
holds for every precomplete numbering.

A classic example of numberings are numberings of the partial 
computable functions. Such a numbering is {\em acceptable\/} 
if it can be effectively translated back and forth into the standard 
numbering of the p.c.\ functions $\vph_e$.
Rogers~\cite{Rogers1967} proved that a numbering is acceptable 
if and only if it satisfies both the enumeration theorem 
and parametrization (also known as the S-m-n--theorem). 
It follows from this that every acceptable numbering is precomplete. 
On the other hand, Friedberg~\cite{Friedberg} showed that there 
exist an effective numbering of the p.c.\ functions without repetitions. 
Friedberg's 1-1 numbering is not precomplete, as can be seen as follows. 
Suppose that $\gamma:\omega\rightarrow \mathcal{P}$ is a 1-1 numbering of 
the (unary) p.c.\ functions that is precomplete. 
By \eqref{precomplete}, we then have for every p.c.\ function $\psi$ 
a computable function $f$ such that
$$
\psi(n)\darrow 
\; \Longrightarrow \; \gamma(f(n))=\gamma(\psi(n))
\; \Longrightarrow \; f(n)= \psi(n). 
$$
The second implication follows because $\gamma$ is 1-1. 
So we see that in fact $f$ is a total extension of~$\psi$. 
But it is well-known that there exist p.c.\ functions that do not 
have total computable extensions. 
So we see that 1-1 numberings of the p.c.\ functions are never precomplete.
For more about 1-1 numberings see Kummer~\cite{Kummer}. 

A topic closely related to numberings is that of computably 
enumerable equivalence relations (ceers). 
For every numbering $\gamma$ we have the equivalence relation 
defined by $n \sim_\gamma m$ if $\gamma(n) = \gamma(m)$. Conversely, 
for every countable equivalence relation, we have the numbering of 
its equivalence classes. Hence the above terminology about 
numberings also applies to ceers. 
Lachlan~\cite{Lachlan}, 
following work of Bernardi and Sorbi~\cite{BernardiSorbi}, 
proved that all precomplete ceers are computably isomorphic.
For a recent survey about ceers, see Andrews, Badaev, and Sorbi~\cite{ABS}.

In the examples of numberings given above, the set $\omega$ 
is not merely a set used to number the elements of a set $S$, 
but it carries extra structure as the domain of the partial 
computable functions, making it into a so-called 
partial combinatory algebra (pca). 
Combinatory completeness is the characteristic property 
that makes a structure with an application operator a pca. 
This is the analog of the S-m-n--theorem (parametrization) for the 
p.c.\ functions. 
In section~\ref{sec:pca} below we review the basic definitions of pca. 

We can extend the notion of numbering from $\omega$ to arbitrary 
pca's as follows. 
A {\em generalized numbering\/} is a surjective mapping 
$\gamma\colon \A \rightarrow S$, where $\A$ is a pca and $S$ is a set. 
This notion was introduced in Barendregt and Terwijn~\cite{BarendregtTerwijn}.
We also have a notion of precompleteness for generalized 
numberings, analogous to Ershov's notion (Definition~\ref{def:precomplete}). 
Precompleteness of generalized numberings is related to the topic
of complete extensions. For example, the identity on a pca $\A$ 
is precomplete if and only if every element of $\A$ (seen as a 
function on $\A$) has a total extension in~$\A$.

In section~\ref{sec:smooth} we show that the numbering of functions of a pca
is precomplete, which is the analog of the precompleteness of the 
standard numbering of the p.c.\ functions.
In general the functions modulo extensional equivalence do not form 
a pca. This prompts the definition of the notion of {\em algebraic\/} 
numbering, which is a generalized numbering that preserves the algebraic 
structure of the pca. 

Below we study generalized numberings in relation to the 
algebraic structure of the pca $\A$. Just as in the lambda-calculus, 
the notion of extensionality is central here. 
A pca is called {\em extensional\/} if $f=g$ whenever 
$fx \simeq gx$ for every~$x$. We have a similar notion of 
extensionality based on generalized numberings 
(Definitions~\ref{def:ext}).
In section~\ref{sec:ext} we show that there is a relation between 
extensionality (an algebraic property) and precompleteness 
(a computability theoretic property).

In section~\ref{sec:strongext} we introduce strong extensionality, 
(Definition~\ref{def:strext}),
and in section~\ref{sec:aux} we introduce some auxiliary equivalence 
relations to aid the comparison between extensional and algebraic.

In section~\ref{sec:algext} we investigate the relations between various 
notions of extensionality and algebraic numberings. 
We will see that neither notion implies the other, 
and that they are in a sense complementary notions.

The Friedberg numbering of the p.c.\ functions quoted above 
also exists for the class of c.e.\ sets. 
In section~\ref{sec:1-1} we discuss the existence of 1-1 numberings 
of classes of sets that are uniformly c.e.\ by considering the 
complexity of equality on those classes. 

Our notation is mostly standard.
In the following, $\omega$ denotes the natural numbers.
$\vph_e$ denotes the $e$-th partial computable
(p.c.) function, in the standard numbering of the p.c.\ functions.
We write $\vph_e(n)\darrow$ if this computation is defined,
and $\vph_e(n)\uarrow$ otherwise.
$W_e = \dom(\vph_e)$ denotes the $e$-th computably enumerable (c.e.) set. 
For unexplained notions from computability theory we refer to 
Odifreddi~\cite{Odifreddi} or Soare~\cite{Soare}.
For background on lambda-calculus we refer to Barendregt~\cite{Barendregt}.

\section{Partial combinatory algebra} \label{sec:pca}

Combinatory algebra predates the lambda-calculus, 
and was introduced by Sch\"{o}nfinkel~\cite{Schoenfinkel}. 
It has close connections with the lambda-calculus, and played an 
important role in its development. 
Partial combinatory algebra (pca) was first studied in 
Feferman~\cite{Feferman}. 
To fix notation and terminology, we will briefly recall the 
definition of a pca, and for a more elaborate treatment refer 
to van Oosten~\cite{vanOosten}.

\begin{definition} \label{def:pca}
A {\em partial applicative structure\/} (pas) is a set $\A$ together
with a partial map from $\A\times \A$ to $\A$.  We denote the image of
$(a,b)$, if it is defined, by $ab$, and think of this as `$a$ applied
to $b$'.  If this is defined we denote this by $ab\darrow$.  By
convention, application associates to the left. We write $abc$ instead
of $(ab)c$.  {\em Terms\/} over $\A$ are built from elements of $\A$,
variables, and application. If $t_1$ and $t_2$ are terms then so is
$t_1t_2$.  If $t(x_1,\ldots,x_n)$ is a term with variables $x_i$, and
$a_1,\ldots,a_n {\in} \A$, then $t(a_1,\ldots,a_n)$ is the term obtained
by substituting the $a_i$ for the~$x_i$.  For closed terms
(i.e.\ terms without variables) $t$ and $s$, we write $t \simeq s$ if
either both are undefined, or both are defined and equal. 
Here application is \emph{strict} in the sense that for $t_1t_2$ to be 
defined, it is required that both $t_1,t_2$ are defined.
We say that an element $f{\in} \A$ is {\em total\/} if $fa\darrow$ for
every $a{\in} \A$.

A pas $\A$ is {\em combinatory complete\/} if for any term
$t(x_1,\ldots,x_n,x)$, $0\leq n$, with free variables among
$x_1,\ldots,x_n,x$, there exists a $b{\in} \A$ such that
for all $a_1,\ldots,a_n,a{\in} \A$,
\begin{enumerate}[\rm (i)]

\item $ba_1\cdots a_n\darrow$,

\item $ba_1\cdots a_n a \simeq t(a_1,\ldots,a_n,a)$.

\end{enumerate}
A pas $\A$ is a {\em partial combinatory algebra\/} (pca) if
it is combinatory complete.
\end{definition}

\begin{theorem} {\rm (Feferman~\cite{Feferman})} \label{Feferman}
A pas $\A$ is a pca if and only if it has elements $k$ and $s$
with the following properties for all $a,b,c\in\A$:
\begin{itemize}

\item $k$ is total and $kab = a$, 

\item $sab\darrow$ and $sabc \simeq ac(bc)$.

\end{itemize}
\end{theorem}

Note that $k$ and $s$ are nothing but partial versions of the 
familiar combinators from combinatory algebra.
As noted in \cite[p95]{Feferman}, Theorem~\ref{Feferman} has the 
consequence that in any pca we can define lambda-terms in the usual 
way (cf.\ Barendregt~\cite[p152]{Barendregt}):\footnote{%
Because the lambda-terms in combinatory algebra do not have
the same substitution properties as in the lambda calculus, 
we use the notation $\lambda^*$ rather than~$\lambda$.
Curry used the notation $[x]$ to distinguish the two 
(cf.\ \cite{CurryFeys}).}  
For every term $t(x_1,\ldots,x_n,x)$, $0\leq n$, with free variables among
$x_1,\ldots,x_n,x$, there exists a term $\lambda^* x.t$ 
with variables among $x_1,\ldots,x_n$,
with the property that for all $a_1,\ldots,a_n,a {\in} \A$,
\begin{itemize}

\item $(\lambda^* x.t)(a_1,\ldots, a_n)\darrow$,

\item $(\lambda^* x.t)(a_1,\ldots, a_n)a \simeq t(a_1,\ldots,a_n,a)$.

\end{itemize}

The most important example of a pca is Kleene's {\em first model\/} $\K_1$, 
consisting of $\omega$ with application defined as $nm = \vph_n(m)$.
Kleene's {\em second model\/} $\K_2$ \cite{KleeneVesley} 
consists of the reals (or more conveniently Baire space $\omega^\omega$), 
with application $\alpha\beta$ defined as applying the continuous 
functional with code $\alpha$ to the real $\beta$. 
See Longley and Normann~\cite{LongleyNormann} for a more detailed definition.

In the axiomatic approach to the theory of computation, there is 
the notion of a basic recursive function theory (BRFT). 
Since this is supposed to model basic computability theory, 
it will come as no surprise that every BRFT gives rise to a pca. 
In case the domain of the BRFT is $\omega$, it actually contains 
a copy of the p.c.\ functions.
See Odifreddi~\cite{Odifreddi} for a discussion of this, 
and references to the literature, including the work of 
Wagner, Strong, and Moschovakis. 

Other pca's can be obtained by relativizing $\K_1$, or by generalizing 
$\K_2$ to larger cardinals. 
Also, every model of Peano arithmetic gives rise to a pca by considering 
$\K_1$ inside the model. 
Further constructions of pca's are discussed in 
van Oosten and Voorneveld~\cite{vanOostenVoorneveld}, and in 
van Oosten~\cite{vanOosten} even more examples of pca's are listed. 
Finally, it is possible to define a combination of Kleene's models
$\K_1$ and $\K_2$ (due to Plotkin and Scott) using enumeration operators, 
cf.\ Odifreddi~\cite[p857ff]{OdifreddiII}.

\section{Generalized numberings}

A {\em generalized numbering\/} \cite{BarendregtTerwijn} is a surjective mapping
$\gamma\colon \A \rightarrow S$, where $\A$ is a pca and $S$ is a set. 
As in the case of ordinary numberings, we have an equivalence relation on 
$\A$ defined by $a \sim_\gamma b$ if $\gamma(a) = \gamma(b)$.

As for ordinary numberings, in principle every generalized 
numbering corresponds to an equivalence relation on $\A$, and 
conversely. However, below we will mainly be interested in 
numberings that also preserve the algebraic structure of the pca, 
making this correspondence less relevant.

The notion of precompleteness for generalized numberings was 
defined in \cite{BarendregtTerwijn}. 
By \cite[Lemma 6.4]{BarendregtTerwijn}, 
the following definition is equivalent to it. 

\begin{definition} \label{def:precomplete}
A generalized numbering $\gamma \colon \A \rightarrow S$ is 
{\em precomplete\/}\footnote{
There is also a notion of completeness for numberings, that we 
will however have no use for in this paper. 
A precomplete generalized numbering $\gamma$ is {\em complete\/} if 
there is a special element $s{\in}S$ (not depending on $b$) 
such that in addition to \eqref{precomplete2},
$\gamma(fa) = s$ for every $a$ with $ba\uarrow$.}
if for every $b{\in} \A$ there exists a total element $f{\in} \A$ 
such that for all $a{\in} \A$, 
\begin{equation} \label{precomplete2}
b{a}\darrow \; \Longrightarrow \; f{a} \sim_\gamma b{a}. 
\end{equation}
In this case, we say that {\em $f$ totalizes $b$ modulo~$\sim_\gamma$\/}.
\end{definition}

In \cite{BarendregtTerwijn}, generalized numberings were used to 
prove a combination of a fixed point theorem for pca's (due to 
Feferman~\cite{Feferman}), and Ershov's recursion theorem \cite{Ershov2} 
for precomplete numberings on~$\omega$. 

Every pca $\A$ has an associated generalized numbering, 
namely the identity $\gamma_\A: \A\rightarrow \A$. 
In section~\ref{sec:ext} we will see examples of when the numbering 
$\gamma_\A$ is or is not precomplete.

\section{Algebraic numberings} \label{sec:smooth}

\begin{definition} \label{def:e}
Let $\A$ be a pca. Define an equivalence on $\A$ by 
$a\sim_e b$ if 
$$
\fa x\in\A (ax \simeq bx).
$$ 
\end{definition}

The following result generalizes the precompleteness of the 
numbering $n\mapsto \vph_n$ of the partial computable functions.

\begin{proposition} \label{prop:precomplete}
The natural map $\gamma_e: \A \rightarrow \A/{\sim_e}$ is precomplete.
Note that $\gamma_e$ is a generalized numbering of the equivalence 
classes.  
\end{proposition}
\begin{proof}
Let $b\in\A$. We have to prove that there is a total $f\in \A$ such 
that when $ba\darrow$ then $fa\sim_e ba$, i.e.\ 
$\fa c{\in}\A (fac \simeq bac)$. 
This follows from the combinatory completeness of $\A$: 
Consider the term $bxy$. By combinatory completeness there exists $f\in A$ 
such that for all $a,c\in \A$, $fa\darrow$ and $fac \simeq bac$. 
\end{proof}

\begin{remark} \label{remark}
Note that $\A/{\sim_e}$ is in general {\em not\/} a pca, at least not 
with the natural application defined by 
$\overline a \cdot \overline b = \overline{a\cdot b}$.
For example, in Kleene's first model $\K_1$ we have for 
$n,m\in \omega$ that $n\sim_e m$ if $\vph_n = \vph_m$. 
Now we can certainly have that $m\sim_e m'$, but 
$\vph_n(m) \neq \vph_n(m')$, so we see that the natural definition
of application $\overline n \cdot \overline m = \overline{n\cdot m}$ in 
$\omega/{\sim_e}$ is not independent of the choice of representative~$m$. 
\end{remark}

The previous considerations prompt the following definition.
First we extend the definition of $\sim_\gamma$ from $\A$ to the set of 
closed terms over $\A$ as follows:

\begin{definition} \label{def:equivalent}
$a \sim_\gamma b$ if either $a$ and $b$ are terms that are both 
undefined, or $a,b\in\A$ and $\gamma(a) = \gamma(b)$.
\end{definition} 
This extended notion $\sim_\gamma$ is the analog of the 
Kleene equality $\simeq$.

\begin{definition} \label{def:smooth}
Call a generalized numbering $\gamma:\A\rightarrow S$ {\em algebraic\/} 
if $\sim_\gamma$ is a congruence, i.e.\ 
$$
a\sim_\gamma a' \wedge b\sim_\gamma b' \Longrightarrow ab \sim_\gamma a'b'.
$$
In this case, we also call the pca $\A$ $\gamma$-algebraic.
\end{definition}

If $\gamma$ is algebraic, we can factor out by $\sim_\gamma$, as in algebra:

\begin{proposition} \label{prop:divide}
Suppose that $\A$ is $\gamma$-algebraic. Then $\A/{\sim_\gamma}$ is again a pca. 
\end{proposition}
\begin{proof}
Define application in $\A/{\sim_\gamma}$ by 
$\overline a\cdot\overline b = \overline{a\cdot b}$.
By algebraicity this is well-defined. 
Combinatory completeness follows because we have the combinators
$\overline s$ and $\overline k$. 
\end{proof}

We note that for a generalized numbering $\gamma:\A\rightarrow S$ 
there is in general no relation between the notion of precompleteness 
and algebraicity. This follows from results in the following 
sections.
\begin{itemize}

\item algebraic does not imply precomplete. 
Namely, let $\gamma_{\K_2}$ be the identity on 
Kleene's second model~$\K_2$. 
Then $\gamma_{\K_2}$ is trivially algebraic. 
However, by \cite{BarendregtTerwijn}, the numbering $\gamma_{\K_2}$
is not precomplete.

\item precomplete does not imply algebraic.
Otherwise, by Proposition~\ref{prop} we would have that
$\A$ is $\gamma$-extensional implies that $\A$ is $\gamma$-algebraic,
contradicting Proposition~\ref{extsmooth}.
Altenatively: The canonical map $\gamma_e: \A \rightarrow \A/{\sim_e}$ 
is precomplete by Proposition~\ref{prop:precomplete}. 
However, it is not algebraic, since otherwise we would have
by Proposition~\ref{prop:divide} that $\A/{\sim_e}$ is a pca,
which in general it is not by Remark~\ref{remark}.

\end{itemize}

\section{Precompleteness and extensionality} \label{sec:ext}

We can think of combinatory completeness of a pca as an analog of the 
S-m-n--theorem (also called the parametrization theorem) 
from computability theory \cite{Odifreddi}. 
Suppose $\A$ is a pca, and $\gamma:\A\rightarrow S$ is a generalized numbering.
Every element $a\in\A$ represents a partial function on $\A$, 
namely $x\mapsto ax$. In analogy to $\K_1$, one could call these the 
partial $\A$-computable functions.  
Note that the precompleteness of the numbering $n\mapsto\vph_n$ of the 
partial computable functions follows from the S-m-n--theorem, 
and that the precompleteness of the numbering of the partial
$\A$-computable functions follows likewise from the 
combinatory completeness of $\A$ (see Proposition~\ref{prop:precomplete}).

Note that the identity on $\K_1$ is not precomplete, 
as there exist p.c.\ functions that do not have a total computable extension. 
In \cite{BarendregtTerwijn} it was shown that the identity 
on Kleene's second model $\K_2$ is also not precomplete.
In general, every pca $\A$ has the identity $\gamma_\A: \A\rightarrow \A$ 
as an associated generalized numbering, and 
$\gamma_\A$ is precomplete if and only if every element $b\in \A$ 
has a total extension $f\in \A$. 
Faber and van Oosten \cite{FabervanOosten} showed that the latter 
is equivalent to the statement that $\A$ is isomorphic to a 
total pca. Here ``isomorphic'' refers to isomorphism in the 
category of pca's introduced in Longley~\cite{Longley}.

\begin{definition} \label{def:ext}
Let $\A$ be a pca, and $\gamma:\A\rightarrow S$ a generalized numbering. 
We say that $\A$ is {\em $\gamma$-extensional\/} if 
\begin{equation} \label{ext}
\fa a\in\A (fa \simeq ga) \Longrightarrow f\sim_\gamma g
\end{equation}
for all $f,g\in \A$.
\end{definition}

In other words, $\A$ is $\gamma$-extensional if 
the relation $\sim_\gamma$ extends the relation $\sim_e$ 
from Definition~\ref{def:e}.
For the special case where $\gamma:\A\rightarrow\A$ is the identity, this is called 
{\em extensionality\/} of $\A$, cf.\ Barendregt~\cite[p1094]{Barendregt1977}.

\begin{proposition} \label{prop}
Suppose $\A$ is $\gamma$-extensional. Then $\gamma$ is precomplete. 
\end{proposition}
\begin{proof}
This is similar to Proposition~\ref{prop:precomplete}.
Given $b\in\A$, we have to prove that there exists a total $f\in \A$ 
such that for every $a\in\A$, $fa\sim_\gamma ba$ whenever $ba\darrow$.
Consider the term $bxy$. 
By combinatory completeness of $\A$ there exists a total $f\in\A$ such that
$fac \simeq bac$ for all $a,c\in\A$. 
Now suppose $ba\darrow$. It follows from 
$\gamma$-extensionality of $\A$ that $fa\sim_\gamma ba$.
Hence $\gamma$ is precomplete.\footnote{
Alternatively, we could derive Proposition~\ref{prop} 
from Proposition~\ref{prop:precomplete} by noticing that 
if $\gamma,\gamma' : \A \rightarrow S$ are generalized numberings 
such that $\sim_{\gamma}$ extends $\sim_\gamma'$, and 
$\gamma'$ is precomplete, then also $\gamma$ is precomplete. 
By Proposition~\ref{prop:precomplete} we have that 
$\gamma' = \gamma_e$ is precomplete, and if $\A$ is $\gamma$-extensional 
then $\sim_\gamma$ extends $\sim_e$, so it follows that 
$\gamma$ is precomplete.}
\end{proof}

In particular, we see from Proposition~\ref{prop} 
that the identity $\gamma_\A$ on $\A$ is 
precomplete if $\A$ is extensional. 

It is possible that a generalized numbering $\gamma:\A\rightarrow S$ 
is precomplete for some other reason than $\A$ being $\gamma$-extensional. 
For example, suppose that $\A$ is a total pca. 
Then $\gamma$ is trivially precomplete
(this is immediate from Definition~\ref{def:precomplete}),
but a total pca $\A$ need not be extensional, as the next proposition shows. 

\begin{proposition} \label{prop:converse}
There exists a generalized numbering $\gamma:\A\rightarrow S$ that 
is precomplete, but such that $\A$ is not $\gamma$-extensional.
\end{proposition}
\begin{proof}
This follows from the fact that there exists a total pca that is not
extensional.
For example, let $\A$ be a model of the lambda calculus. This certainly 
does not have to be extensional, for example the graph model 
$P\omega$ is not extensional, cf.\ \cite[p474]{Barendregt}. 
(An example of a model of the lambda calculus that {\em is\/} extensional 
is Scott's model $D_\infty$.)

Another example for is the set of terms $\mathfrak{M}(\beta\eta)$ 
in the lambda calculus. 
This combinatory algebra is extensional, by inclusion of the $\eta$-rule.  
However, the set of {\em closed\/} terms $\mathfrak{M}^0(\beta\eta)$ is 
not extensional by Plotkin~\cite{Plotkin1974}. 
\end{proof}

Proposition~\ref{prop:converse} shows that the converse of 
Proposition~\ref{prop} does not hold.

By the results of \cite{BarendregtTerwijn}, combinatory completeness 
of $\A$ does not imply precompleteness of $\gamma_\A$. 
Conversely, precompleteness of $\gamma_\A$ also does not imply 
combinatory completeness of $\A$. 
Namely, simply take any total applicative system $\A$ that is not 
a pca.\footnote{An example of such a system is the set 
$\A=\omega^{<\omega}$ of finite strings of numbers, with concatenation 
of strings as application. This is a pure term model, 
in which application of terms can only increase the length of terms. 
This implies that this pas is not combinatory complete, as there 
cannot be a combinator $k$, which does reduce the length of terms.} 
Note that the identity $\gamma_\A:\A\rightarrow\A$ is precomplete, 
as the pas $\A$ is total.

\section{Strong extensionality}\label{sec:strongext}

Given a generalized numbering $\gamma:\A\rightarrow S$, 
we have two kinds of equality on $\A$: $a\simeq b$ and $a\sim_\gamma b$. 
The notion of $\gamma$-extensionality is based on the former. 
We obtain a stronger notion if we use the latter, where we use 
the extended notion of $\sim_\gamma$ for closed terms from 
Definition~\ref{def:equivalent}.

\begin{definition} \label{def:strext}
We call $\A$ {\em strongly $\gamma$-extensional\/} if 
$$
\fa x (fx \sim_\gamma gx) \Longrightarrow f\sim_\gamma g
$$
for every $f,g\in\A$.
\end{definition}

\begin{theorem} \label{strictext}
Strong $\gamma$-extensionality implies $\gamma$-extensionality,
but not conversely. 
\end{theorem}
\begin{proof}
First, strong $\gamma$-extensionality implies $\gamma$-extensionality 
because $\fa x \; fx \simeq gx$ implies 
$\fa x \; fx \sim_\gamma gx$, so the premiss 
of the first notion is weaker than that of the second. 

To see that the implication is strict, we exhibit a pca $\A$ that 
is $\gamma$-extensional but not strongly $\gamma$-extensional.
Take $\A$ to be Kleene's first model $\K_1$, and let $\gamma = \gamma_e$, 
the numbering of equivalence classes from Proposition~\ref{prop:precomplete}.
That $\A$ is $\gamma_e$-extensional is trivial, since 
$f\sim_e g$ means precisely $\fa x \; fx\simeq gx$.

To see that $\A$ is not strongly $\gamma_e$-extensional, 
let $d,e\in \A$ be such that $d\neq e$ and 
$\fa x \; dx\simeq ex$, and define 
$f = kd$ and $g=ke$, with $k$ the combinator. 
Then $fx=d$ and $gx=e$, hence 
$\fa x \; fx \sim_e gx$ because $d\sim_e e$.
But not $\fa x \; fx \simeq gx$ because $d\neq e$, 
so $f \not\sim_e g$.
\end{proof}

\section{Left and right equivalences}\label{sec:aux}

For the discussion below (and also to aid our thinking), 
we introduce two equivalence relations. 
Let $\A$ be a pca, and $\gamma:\A\rightarrow S$ a generalized numbering.

\begin{definition} 
We define two kinds of equivalence relations on $\A$, 
corresponding to right and left application:  
\begin{itemize} 
\item \makebox[1.45cm][l]{$f\sim_{R_\gamma} g$} if $\fa x \; fx \sim_\gamma gx$.  
\item \makebox[1.45cm][l]{$f\sim_{L_\gamma} g$} if $\fa z \; zf \sim_\gamma zg$. 
\end{itemize} 
\end{definition}

Recall the relation $\sim_e$ from section~\ref{sec:smooth}.
Note that $\sim_e$ is equal to $\sim_{R_\gamma}$ for 
$\gamma$ the identity. 
With these equivalences we can succinctly express extensionality 
as follows:
\begin{align*}
\text{$\A$ is $\gamma$-extensional} & 
\text{ if $f \sim_e g \Longrightarrow f\sim_\gamma g$,} \\
\text{$\A$ is strongly $\gamma$-extensional} & 
\text{ if $f \sim_{R_\gamma} g \Longrightarrow f\sim_\gamma g$.}  
\end{align*}

\begin{proposition} \label{propLR}
For all $f,g\in\A$ we have
$f\sim_{L_\gamma} g \Longrightarrow f\sim_{R_\gamma} g$.
\end{proposition}
\begin{proof}
For every $x$, define $z_x = \lambda^* h. hx$. 
(Note that in every pca we can define such lambda terms, 
cf.\ section~\ref{sec:pca}.) 
Then 
\begin{align*}
f \sim_{L_\gamma} g &\Longrightarrow \fa x\; z_x f \sim_\gamma z_x g \\
&\Longrightarrow \fa x\; fx \sim_\gamma gx \\
&\Longrightarrow f \sim_{R_\gamma} g.
\qedhere
\end{align*}
\end{proof}

\section{Algebraic versus extensional} \label{sec:algext}

For a given pca $\A$ and a generalized numbering $\gamma$ on $\A$, 
note that the following hold:
If $\A$ is $\gamma$-algebraic, then for every $f,g\in\A$, 
\begin{align}
f\sim_\gamma g &\Longrightarrow f\sim_{R_\gamma} g, \label{s1} \\
f\sim_\gamma g &\Longrightarrow f\sim_{L_\gamma} g. \label{s2}
\end{align}
This holds because in Definition~\ref{def:smooth}, 
we can either take the right sides equal, obtaining \eqref{s1}, 
or take the left sides equal, obtaining~\eqref{s2}. 
Also note that \eqref{s2} actually implies \eqref{s1} 
by Proposition~\ref{propLR}.  
We could call \eqref{s1} {\em right-algebraic\/} and
\eqref{s2} {\em left-algebraic}.

\begin{proposition}
$\gamma$-algebraic is equivalent with \eqref{s2}.
\end{proposition}
\begin{proof}
That \eqref{s2} follows from $\gamma$-algebraicity was noted above. 
Conversely, assume \eqref{s2} and 
suppose that $a\sim_\gamma a'$ and $b\sim_\gamma b'$. 
We have to prove that $ab \sim_\gamma a'b'$. Indeed we have
\begin{align*}
ab &\sim_\gamma a'b  &&\text{by \eqref{s1}}   \\
   &\sim_\gamma a'b' &&\text{by \eqref{s2}}.
\qedhere
\end{align*}
\end{proof}

On the other hand, if $\A$ is strongly $\gamma$-extensional, we have 
\begin{align}
f\sim_{R_\gamma} g &\Longrightarrow f\sim_\gamma g \label{e3} 
\end{align}
which is the converse of \eqref{s1}.
So we see that in a sense, the notions of algebraicity and 
extensionality are complementary.
We now show that neither of them implies the other. 

\begin{proposition} \label{extsmooth}
$\gamma$-extensional does not imply $\gamma$-algebraic.
\end{proposition}
\begin{proof}
Consider Kleene's first model $\K_1$, and let $\gamma=\gamma_e$ be 
the numbering from Proposition~\ref{prop:precomplete}.
Every pca is trivially $\gamma_e$-extensional, 
as $f\sim_e g \Rightarrow f\sim_e g$.
However, $\K_1$ is not $\gamma_e$-algebraic. Namely, \eqref{s2} above
does not hold: There are $n,m\in \K_1$ such that 
$n\sim_e m$, i.e.\ $n$ and $m$ are codes of the same partial computable 
function, but $n\neq m$, so that $\fa z \; zn\sim_e zm$ does not hold:
There is a p.c.\ function $\vph$ such that $\vph(n) \not\sim_e \vph(m)$. 
\end{proof}

We can strengthen Proposition~\ref{extsmooth} to the following.

\begin{theorem}  \label{strongextsmooth}
Strong $\gamma$-extensional does not imply $\gamma$-algebraic.
\end{theorem}
\begin{proof}
We show that \eqref{e3} does not imply \eqref{s2}.
As a pca we take Kleene's first model $\K_1$, and we define a 
generalized numbering $\gamma$ on it as follows. 
We start with the equivalence $\sim_e$ on $\K_1$, 
and we let $\sim_\gamma$ be the smallest extension of $\sim_e$
such that 
\begin{equation} \label{fixpoint}
f\sim_\gamma g \Longleftrightarrow \fa x \; fx\sim_\gamma gx.
\end{equation}
(Here we read $fx\sim_\gamma gx$ as in Definition~\ref{def:equivalent}.)
The equivalence relation $\sim_\gamma$ is the smallest fixed point of 
the monotone operator that, given an equivalence $\sim$ on $\K_1$ 
that extends $\sim_e$, defines a new equivalence $\approx$ by 
\begin{equation} \label{operator}
f\approx g \Longleftrightarrow \fa x \; fx \sim gx.
\end{equation}
The existence of $\sim_\gamma$ is then guaranteed by the Knaster-Tarski 
theorem on fixed points of monotone operators~\cite{KnasterTarski}.
Note that by Remark~\ref{remark}, $\K_1/{\sim_e}$ is not a pca, 
and neither are the extensions $\K_1/{\approx}$, but this is not 
a problem for the construction \eqref{operator}, since the 
application $fx$ keeps taking place in the pca~$\K_1$.
Note that by \eqref{fixpoint} we have that 
$\K_1$ is strongly $\gamma$-extensional.

We claim that \eqref{s2} fails for $\gamma$, and hence that 
$\K_1$ is not $\gamma$-algebraic. 
First we observe that $\gamma$ is not trivial, i.e.\ 
does not consist of only one equivalence class. 
Namely, let $ax\uarrow$ for every~$x$, and let $b\in \K_1$ 
be total. Then obviously $\fa x \; ax\sim_\gamma bx$ does not hold,
hence by \eqref{fixpoint} we have $a\not\sim_\gamma b$.

For the failure of \eqref{s2} we further need the existence of 
$f\sim_\gamma g$ such that $f\neq g$. 
Such $f$ and $g$ exist, since they already exist for $\sim_e$, 
and $\sim_\gamma$ extends~$\sim_e$.
Now let $a\not\sim_\gamma b$ (the existence of which we noted above), 
and let $z$ be a code of a partial computable function such that 
$zf=a$ and $zg=b$. 
Then $zf\not\sim_\gamma zg$, hence $\fa z \; zf\sim_\gamma zg$ does 
not hold, and thus \eqref{s2} fails. 
\end{proof}

\begin{corollary} \label{cor}
\eqref{s2} implies \eqref{s1}, but not conversely.
\end{corollary}
\begin{proof}
\eqref{s2} implies \eqref{s1} by Proposition~\ref{propLR}.
In the counterexample of Theorem~\ref{strongextsmooth}
the equivalence \eqref{fixpoint} holds,
so both \eqref{s1} and \eqref{e3} hold, but \eqref{s2} does not.
\end{proof}

\begin{proposition}
$\gamma$-algebraic does not imply $\gamma$-extensional.
\end{proposition}
\begin{proof}
Let $\A$ be a pca that is not extensional (such as Kleene's $\K_1$), 
and let $\gamma$ be the identity on $\A$. 
Every pca is always $\gamma$-algebraic, so this provides a 
counterexample to the implication.
\end{proof}

The proof of Theorem~\ref{strongextsmooth} shows that 
\eqref{e3} does not imply \eqref{s2}. 
Since \eqref{s2} strictly implies \eqref{s1} by Corollary~\ref{cor}, 
a stronger statement would be to show that 
\eqref{e3} does not imply~\eqref{s1}. We can obtain this with a 
variation of the earlier proof (though this construction
is less natural, and only serves a technical purpose).

\begin{theorem}
Strong $\gamma$-extensional \eqref{e3} does not imply 
right-al\-geb\-raic~\eqref{s1}.
\end{theorem}
\begin{proof}
We show that \eqref{e3} does not imply \eqref{s1}.
As before, we take Kleene's first model $\K_1$, and we define a
generalized numbering $\gamma$ on it.
Now we do not want the equivalence \eqref{fixpoint} to hold, 
so we do not start with the equivalence relation $\sim_e$.

Let $f,g\in\K_1$ and $x\in\omega$ be such that 
$fx\darrow$ is total, and $gx\uarrow$.
Start with $f \sim_\gamma g$, so that $\gamma$ equates $f$ and $g$
and nothing else. 
Let $\sim_\gamma$ be the smallest extension of this equivalence 
relation such that
\begin{equation*} 
\fa x (fx\sim_\gamma gx) \Longrightarrow f\sim_\gamma g,
\end{equation*}
where again, we read $fx\sim_\gamma gx$ as in 
Definition~\ref{def:equivalent}.
As before, the equivalence relation $\sim_\gamma$ exists by the 
Knaster-Tarski theorem~\cite{KnasterTarski}.
This ensures that $\gamma$ satisfies~\eqref{e3}.

We claim that $\fa x \; fx\sim_\gamma gx$ does not hold, 
and hence that \eqref{s1} fails.
Note that we would only have $\fa x \; fx\sim_\gamma gx$ if 
at some stage $\fa x,y \; fxy\sim_\gamma gxy$ would hold. 
However, since $fx$ is total, $fxy$ is always defined, 
whereas $gxy$ is never defined by choice of~$x$.
Hence, by Definition~\ref{def:equivalent}, $fxy\not\sim_\gamma gxy$, 
whatever $\gamma$ may be.
\end{proof}

Consider the following property, which is the converse of 
the implication from Proposition~\ref{propLR}:
\begin{equation}
f\sim_{R_\gamma} g \Longrightarrow f\sim_{L_\gamma} g. \label{eiggamma}
\end{equation}
This property expresses that whenever $f$ and $g$ denote the
same function, they are {\em inseparable\/} in the pca 
(compare Barendregt~\cite[p48]{Barendregt}).
Combining algebraicity and extensionality, we obtain the following relation. 

\begin{proposition} \label{converse}
$\gamma$-algebraic + strongly $\gamma$-extensional $\Longrightarrow$ \eqref{eiggamma}.
\end{proposition}
\begin{proof} 
By $\gamma$-algebraicity we have \eqref{s2}, hence 
\begin{align*}
f\sim_{R_\gamma} g &\Longrightarrow f\sim_\gamma g &&\text{by strong $\gamma$-extensionality}\\
&\Longrightarrow f\sim_{L_\gamma} g &&\text{by \eqref{s2}.}
\qedhere
\end{align*}
\end{proof}

Note that the identity $\gamma_\A$ is algebraic for any pca $\A$. 
So any extensional $\A$ (meaning $\gamma_\A$-extensional, 
which in this case coincides with strongly $\gamma_\A$-extensional)
is an example where the conditions of Proposition~\ref{converse} hold.

\section{A note on 1-1 numberings of uniformly c.e.\ classes}\label{sec:1-1}

A class of c.e.\ sets $\LL$ is called {\em uniformly c.e.\/} if 
it admits a computable numbering of its indices, i.e.\ if it 
is of the form 
$$
\LL=\{W_{f(n)}:n\in\omega\}
$$ 
with $f$ a computable function. 
Such a numbering $f$ is called 1-1 if the numbering does not have any 
repetitions, i.e.\ if the sets $W_{f(n)}$ are all different.
A classic question is which uniformly c.e.\ classes admit a \mbox{1-1} numbering. 
For a list of references about this topic see 
Odifreddi~\cite[p228 ff.]{Odifreddi} and Kummer \cite{Kummer}. 

Friedberg~\cite{Friedberg} showed that the class of all c.e.\ sets 
has a 1-1 numbering. 
The difficulty of course lies in the fact that equality for c.e.\ sets
is hard to decide: The set $\{(n,m)\mid W_n=W_m\}$ 
is $\Pi^0_2$-complete, cf.\ Soare~\cite{Soare}.
Here we prove two results relating the problem above to the 
complexity of the equality relation. 

\begin{proposition}
If for an infinite uniformly c.e.\ class $\LL$ the equality relation is 
$\Pi^0_1$ then $\LL$ has a 1-1 numbering. 
\end{proposition}
\begin{proof}
Let $f$ be computable such that $\LL=\{W_{f(n)}:n\in\omega\}$. 
The statement that the equality relation of $\LL$ is $\Pi^0_1$ 
means that the set 
$$
U=\{(n,m) \mid W_{f(n)}\neq W_{f(m)}\}
$$
is c.e. Enumerate $\LL$ as follows. Enumerate $f(n)$ in $\LL$ if 
and only if $\mbox{$(\forall m<n)$}$ $\mbox{$[(n,m)\in U]$}$. 
This is clearly a computable enumeration because $U$ is c.e. 
Also, it is easy to see that for every set in $\LL$, 
$f(n)$ is enumerated into $\LL$ 
for the minimal $n$ such that $f(n)$ is a code for this set. This 
proves that the enumeration is a 1-1 numbering. 
\end{proof}

\begin{proposition}
A uniformly c.e.\ class $\LL$ for which the equality relation is
$\Sigma^0_1$ does not necessarily have a 1-1 numbering. 
\end{proposition}
\begin{proof}
Define a uniformly c.e.\ class $\LL$ as follows. 
We want to ensure that $\vph_e$ is not a 1-1 enumeration of $\LL$.
To this end, for every $e$ we put two codes $x_e$ and $y_e$ of c.e.\ sets 
into $\LL$, with $W_{y_e}=\{2e, 2e+1\}$. 
The code $x_e$ is defined by the following enumeration procedure 
for $W_{x_e}$. Enumerate $2e$ into $W_{x_e}$. Search for two 
different codes $a$, $b$ in the range of $\varphi_e$ such that 
$2e\in W_a$ and $2e\in W_b$. If such $a$ and $b$ are found enumerate 
$2e+1$ in $W_{x_e}$. The class $\LL$ thus defined is uniformly c.e.\ 
and has a $\Sigma^0_1$ equality relation, because to find out whether 
two sets in $\LL$ are equal, enumerate them and compare the first 
even elements enumerated (every set in $\LL$ contains exactly one 
even element). If these are equal, say they both equal $2e$, 
then we know that the two sets are $W_{x_e}$ and $W_{y_e}$, and 
these are equal if and only if the $\Sigma^0_1$-event from the 
enumeration procedure of $W_{x_e}$ occurs. 

To prove that no computable function 1-1 enumerates $\LL$, 
fix $e$ such that $\varphi_e$ is total. Now if 
two different codes $a$ and $b$ occur in the range of $\varphi_e$ such 
that $2e\in W_a\cap W_b$, then $W_{x_e}=W_{y_e}$ by 
definition of $x_e$, 
hence $\varphi_e$ does not one-one enumerate $\LL$. In the 
case that such $a$ and $b$ do not appear in ${\rm range}(\varphi_e)$
it holds that $W_{x_e}\neq W_{y_e}$, hence $\varphi_e$ does not 
enumerate all the elements of $\LL$. 
\end{proof}

\end{document}